\documentclass[reqno,11pt]{article} 
\usepackage[left=2.5cm,right=2.5cm,top=3cm,bottom=3cm,a4paper]{geometry}
\usepackage{amsmath, amssymb}
\usepackage{amsthm, amscd} 
\usepackage[all,cmtip]{xy}
\usepackage{float}
\usepackage{caption}
\usepackage{color}
\usepackage{rotating}
\usepackage[utf8]{inputenc}
\usepackage[T1]{fontenc}
\usepackage{hyperref}

\makeatletter
\DeclareRobustCommand{\rvdots}{%
  \vbox{
    \baselineskip4\p@\lineskiplimit\z@
    \kern-\p@
    \hbox{.}\hbox{.}\hbox{.}
  }}
\makeatother

\newcommand{\Mod}[1]{\ (\textup{mod}\ #1)}
\makeatletter
\def\moverlay{\mathpalette\mov@rlay}
\def\mov@rlay#1#2{\leavevmode\vtop{%
   \baselineskip\z@skip \lineskiplimit-\maxdimen
   \ialign{\hfil$\m@th#1##$\hfil\cr#2\crcr}}}
\newcommand{\charfusion}[3][\mathord]{
    #1{\ifx#1\mathop\vphantom{#2}\fi
        \mathpalette\mov@rlay{#2\cr#3}
      }
    \ifx#1\mathop\expandafter\displaylimits\fi}
\makeatother

\newcommand{\cupdot}{\charfusion[\mathbin]{\cup}{\cdot}}
\theoremstyle{plain} 
\newtheorem{theorem}{\indent\sc Theorem}[section]
\newtheorem{lemma}[theorem]{\indent\sc Lemma}

\newtheorem{proposition}[theorem]{\indent\sc Proposition}

\theoremstyle{definition} 

\newtheorem{remark}[theorem]{\indent\sc Remark}
\newtheorem{example}[theorem]{\indent\sc Example}

\makeatletter
\def\address#1#2{\begingroup
\noindent\parbox[t]{7.8cm}{%
\small{\scshape\ignorespaces#1}\par\vskip1ex
\noindent\small{\itshape E-mail address}%
\/: #2\par\vskip4ex}\hfill%
\endgroup}%
\makeatother
 
\title{A simplified algorithmic realization of Galois actions
on special values of modular functions}
\author{
\textsc{Ja Kyung Koo, Dong Hwa Shin and Dong Sung Yoon$^*$} 
}
\date{} 
\begin{document}

\allowdisplaybreaks

\maketitle

\begin{abstract}
We propose an explicit and practical algorithm for computing Galois conjugates and irreducible polynomials
for special values of modular functions evaluated at CM points associated with imaginary quadratic orders.
Our approach builds upon the theory of extended form class groups developed by Jung et al., 
offering a refinement of earlier methods by Stevenhagen and Cho, respectively. 
\end{abstract}

\tableofcontents

\footnote{ 
2020 \textit{Mathematics Subject Classification}. Primary 11R37; Secondary 11E12, 11F03, 11R65.}
\footnote{ 
\textit{Key words and phrases}. Class field theory, form class groups, ideal class groups,
modular functions.} \footnote{$^*$Corresponding author.\\
\thanks{The second named author was supported
by Hankuk University of Foreign Studies Research Fund of 2025 and by the National Research Foundation of Korea (NRF) grant funded by the Korea government (MSIT)
(No. RS-2023-00241953). The third (corresponding) named author was supported
by the National Research Foundation of Korea (NRF) grant funded by the Korea government (MSIT)
(No. RS-2024-00342522). 
}
}

\section {Introduction}

Let $\mathbb{H}=\{\tau\in\mathbb{C}~|~\mathrm{Im}(\tau)>0\}$ denote the complex upper half-plane. 
In \cite{Duke}, Duke interpreted the Rogers-Ramanujan continued fraction
\begin{equation*}
r(\tau)=\cfrac{q^{1/5}}{1+\cfrac{q}{1+\cfrac{q^2}{1+\cfrac{q^3}{1+\cdots\phantom{\Bigg|}}}}}\qquad(\tau\in\mathbb{H},~q=e^{2\pi\mathrm{i}\tau})
\end{equation*}
as a meromorphic modular function, specifically a Hauptmodul for the principal congruence subgroup $\Gamma(5)$. 
He also proved that $x=r(\tau)$ ($\tau\in\mathbb{H}$) satisfies the icosahedral equation 
\begin{equation}\label{icosahedral}
(x^{20}-228x^{15}+494x^{10}+228x^5+1)^3+j(\tau)x^5(x^{10}+11x^5-1)^5=0,
\end{equation}
where $j=q^{-1}+744+196884q+21493760q^2+\cdots$ is the classical elliptic modular function. 
It follows that a special value $r(\tau)$ can be expressed in terms of radicals over $\mathbb{Q}$ if and only if $j(\tau)$ can be, 
and the icosahedral equation is reducible over $\mathbb{Q}(\zeta_5,\,j(\tau))$ where $\zeta_5=e^{2\pi\mathrm{i}/5}$ (\cite[Theorem 1]{Duke}). 
In particular, if $\tau$ is imaginary quadratic, then  
$r(\tau)$ is a unit that can be expressed in terms of radicals over $\mathbb{Q}$ (\cite[Theorem 2]{Duke}). 
See \cite{C-C} for further discussion.
\par
Inspired by \cite{Duke}, this paper is concerned with the special values of meromorphic modular functions. 
Let $K$ be an imaginary quadratic field with ring of integers $\mathcal{O}_K$,
and let $E$ be an elliptic curve with $j$-invariant $j_E=j(\mathcal{O}_K)$, defined over $\mathbb{Q}(j_E)$. 
Let $N$ be a positive integer. 
According to the theory of complex multiplication, primarily established by Artin, Hasse and Hecke, 
the ray class field $K_{(N)}$ of $K$ modulo $N\mathcal{O}_K$ is generated by 
the values of $j_E$ and $h_E(E[N])$, where $h_E$ denotes the Weber function associated with $E$
and $E[N]$ is the group of $N$-torsion points on $E$ (see \cite[Theorem 2 in Chapter 10]{Lang}). 
Furthermore, by the work of Shimura \cite[Chapter 6]{Shimura}, one obtains
\begin{equation}\label{K_N}
K_{(N)}=K(f(\tau_K)~|~f\in\mathcal{F}_N~\textrm{is finite at}~\tau_K), 
\end{equation}
where $\tau_K\in\mathbb{H}$
satisfies $\mathcal{O}_K=\mathbb{Z}\tau_K+\mathbb{Z}$, and $\mathcal{F}_N$ is the field of meromorphic modular functions of level $N$ whose Fourier 
expansions with respect to $q^{1/N}$ have coefficients in the $N$th cyclotomic field 
(cf. \cite[Corollary to Theorem 2 in Chapter 10]{Lang}).  
\par
On the other hand, more generally, for an order $\mathcal{O}$ in $K$ with discriminant $D$
we consider
the associated ray class field $K_{\mathcal{O},\,N}$ modulo $N\mathcal{O}$,
which was introduced by S\"{o}hngen \cite{Sohngen} and
later reformulated by Stevenhagen \cite{Stevenhagen} using the language of adeles and ideles. 
As shown in \cite[Theorem 4]{Cho}, we have 
\begin{equation*}
K_{\mathcal{O},\,N}=K(f(\tau_\mathcal{O})~|~f\in\mathcal{F}_N~\textrm{is finite at}~\tau_\mathcal{O}), 
\end{equation*}
where $\tau_\mathcal{O}$ is defined analogously to $\tau_K$\,; see (\ref{tau_O}). Although Cho \cite[Theorem 8]{Cho} provided
an explicit description of all Galois conjugates of $f(\tau_\mathcal{O})$ over $K$ by 
using ideas from \cite[Sections 3, 4, 6]{Stevenhagen} (cf. \cite[Theorems 15.22 and 15.27]{Cox}), 
the construction remains unwieldy in practice. 
\par 
The purpose of this paper is to refine the approaches of Stevenhagen and Cho, and 
to present a simplified and practical algorithm for computing 
all Galois conjugates of $f(\tau_\mathcal{O})$, as well as its irreducible polynomial over $\mathbb{Q}$
(Theorems \ref{main} and \ref{algorithm}). 
To this end, we utilize the extended form class group of discriminant $D$ and level $N$,
recently developed by Jung-Koo-Shin-Yoon \cite{J-K-S-Y23}. 
This group is isomorphic to the ray class group of order $\mathcal{O}$ modulo $N\mathcal{O}$,
generalizing Gauss's form class group $\mathcal{C}(D)$ (see \cite[Sections 3 and 7]{Cox} and 
Sections \ref{sect3} and \ref{sect4} of this paper). 

\section {Ray class fields of orders}\label{sect2}

Throughout this paper, we set\,:
\begin{itemize}
\item $D$ : a negative integer such that $D\equiv0$ or $1\Mod{4}$,
\item $K$ : the imaginary quadratic field $\mathbb{Q}(\sqrt{D})$,
\item $\mathcal{O}_K$ : the ring of integers of $K$,
\item $\mathcal{O}$ : the order in $K$ with discriminant $D$,
\item $\ell_\mathcal{O}$ : the conductor of $\mathcal{O}$, i.e., $\ell_\mathcal{O}=[\mathcal{O}_K:\mathcal{O}]$, 
\item $N$ : a positive integer. 
\end{itemize}
In this preliminary section, we shall describe the ray class field of order $\mathcal{O}$ modulo $N\mathcal{O}$,
following the perspective of \cite[Chapter 2]{Cox} 
(see also \cite[Chapter 6]{Schertz}, and \cite{Stevenhagen} for an adelic reformulation). 
\par
For each positive integer $m$, 
define the following two subgroups of proper fractional $\mathcal{O}$-ideals\,:
\begin{align*}
I(\mathcal{O},\,m)&=\langle\textrm{nonzero proper $\mathcal{O}$-ideals prime to $m$}\rangle,\\
P_N(\mathcal{O},\,m)&=
\langle\nu\mathcal{O}~|~\textrm{$\nu$ is a nonzero element of $\mathcal{O}$ such that}~
~\nu\mathcal{O}~\textrm{is prime to $m$ and}~
\nu\equiv1\Mod{N\mathcal{O}}\rangle,
\end{align*}
and let
\begin{equation*}
\mathcal{C}_N(\mathcal{O},\,m)=I(\mathcal{O},\,m)/P_N(\mathcal{O},\,m)
\end{equation*}
be the associated quotient group. 
Then we obtain an isomorphism 
\begin{equation}\label{CIP}
\mathcal{C}_N(\mathcal{O},\,\ell_\mathcal{O}N)\stackrel{\sim}{\rightarrow}
I(\mathcal{O}_K,\,\ell_\mathcal{O}N)/P_{\mathbb{Z},\,N}(\mathcal{O}_K,\,\ell_\mathcal{O}N),
\quad[\mathfrak{a}]\mapsto[\mathfrak{a}\mathcal{O}_K],
\end{equation}
where
\begin{equation*}
P_{\mathbb{Z},\,N}(\mathcal{O}_K,\,\ell_\mathcal{O}N)=\left\langle
\nu\mathcal{O}_K~\Bigg|~
\begin{array}{l}
\textrm{$\nu$ is a nonzero element of $\mathcal{O}_K$ satisfying that $\nu\mathcal{O}_K$ is prime to $\ell_\mathcal{O}N$}\\
\textrm{and}~\nu\equiv a\Mod{\ell_\mathcal{O}N\mathcal{O}_K}~\textrm{for some}~a\in\mathbb{Z}~
\textrm{such that}~a\equiv1\Mod{N}
\end{array}
\right\rangle 
\end{equation*}
(\cite[Proposition 2.8]{J-K-S-Y23}). 
We simply write $\mathcal{C}_N(\mathcal{O})$ for $\mathcal{C}_N(\mathcal{O},\,N)$, 
and refer to it as
the \textit{ray class group of $\mathcal{O}$ modulo $N\mathcal{O}$}. 
Since the inclusion $I(\mathcal{O},\,\ell_\mathcal{O}N)\hookrightarrow
I(\mathcal{O},\,N)$ induces an isomorphism
\begin{equation*}
\mathcal{C}_N(\mathcal{O},\,\ell_\mathcal{O}N)\stackrel{\sim}{\rightarrow}\mathcal{C}_N(\mathcal{O})
\end{equation*}
(\cite[Proposition 2.13]{J-K-S-Y23}),  via the isomorphism in (\ref{CIP}), we attain the isomorphism 
\begin{equation}\label{CaIP}
\mathcal{C}_N(\mathcal{O})\stackrel{\sim}{\rightarrow}
I(\mathcal{O}_K,\,\ell_\mathcal{O}N)/P_{\mathbb{Z},\,N}(\mathcal{O}_K,\,\ell_\mathcal{O}N).
\end{equation}
Therefore, by the existence theorem of class field theory
(see \cite[Theorem 8.6]{Cox} or \cite[Section V.7]{Janusz}), 
there is a unique abelian extension of $K_{\mathcal{O},\,N}$ of $K$ satisfying\,:
\begin{enumerate}
\item[(i)] Every nonzero prime ideal of $\mathcal{O}_K$ that ramifies in $K_{\mathcal{O},\,N}$
divides $\ell_\mathcal{O}N\mathcal{O}_K$,
\item[(ii)] The isomorphism in (\ref{CaIP}),
together with the Artin reciprocity map for the modulus $\ell_\mathcal{O}N\mathcal{O}_K$, induces
an isomorphism
\begin{equation*}
\mathcal{C}_N(\mathcal{O})
\stackrel{\sim}{\rightarrow}\mathrm{Gal}(K_{\mathcal{O},\,N}/K). 
\end{equation*}
\end{enumerate}
The field $K_{\mathcal{O},\,N}$ is called the 
\textit{ray class field of order $\mathcal{O}$ modulo $N\mathcal{O}$}. 

\section {Extended form class groups}\label{sect3}

Generalizing Gauss's form class group of discriminant $D$, 
Jung et al. \cite{J-K-S-Y23} constructed the extended form class group $\mathcal{C}_N(D)$
of discriminant $D$ and level $N$.
We shall briefly introduce $\mathcal{C}_N(D)$ as a group isomorphic to the ray class group $\mathcal{C}_N(\mathcal{O})$
of order $\mathcal{O}$ modulo $N\mathcal{O}$.  
\par
Let $\mathcal{Q}(D,\,N)$ denote the set of primitive positive definite binary quadratic forms over $\mathbb{Z}$ of discriminant $D$
satisfying an additional condition\,:
\begin{equation*}
\mathcal{Q}(D,\,N)=\left\{
Q(x,\,y)=Q\left(\begin{bmatrix}x\\y\end{bmatrix}\right)=ax^2+bxy+cy^2\in\mathbb{Z}[x,\,y]~\Bigg|~
\begin{array}{l}
\gcd(a,\,b,\,c)=1,\\
b^2-4ac=D,~a>0,\\
\gcd(a,\,N)=1
\end{array}
\right\}.
\end{equation*}
The congruence subgroup 	
\begin{equation*}
\Gamma_1(N)=\left\{\gamma\in\mathrm{SL}_2(\mathbb{Z})~|~
\gamma\equiv\begin{bmatrix}1&\mathrm{*}\\0&1\end{bmatrix}\Mod{NM_2(\mathbb{Z})}
\right\}
\end{equation*}
has a right action on the set $\mathcal{Q}(D,\,N)$, which induces the equivalence relation 
$\sim_{\Gamma_1(N)}$ defined as follows\,: for $Q,\,Q'\in\mathcal{Q}(D,\,N)$,
\begin{equation*}
Q\sim_{\Gamma_1(N)}Q'\quad\Longleftrightarrow\quad
Q'\left(\begin{bmatrix}x\\y\end{bmatrix}\right)=
Q\left(\begin{bmatrix}x\\y\end{bmatrix}\right)^\gamma=Q\left(\gamma\begin{bmatrix}x\\y\end{bmatrix}\right)~\textrm{for some}~\gamma\in\Gamma_1(N). 
\end{equation*}
Let 
\begin{equation*}
\mathcal{C}_N(D)=\mathcal{Q}(D,\,N)/\sim_{\Gamma_1(N)}
\end{equation*}
be the set of equivalence classes. 
For each $Q(x,\,y)=ax^2+bxy+cy^2\in\mathcal{Q}(D,\,N)$, let $\omega_Q$ be the root 
of the quadratic polynomial $Q(x,\,1)=ax^2+bx+c$ that lies in the complex upper half-plane
$\mathbb{H}$, namely,
\begin{equation*}
\omega_Q=\frac{-b+\sqrt{D}}{2a}.
\end{equation*}

\begin{proposition}\label{CDCO}
The set $\mathcal{C}_N(D)$ can be endowed with a group structure such that the map
\begin{equation*}
\mathcal{C}_N(D)\rightarrow\mathcal{C}_N(\mathcal{O}),\quad [Q]\mapsto[\mathbb{Z}\omega_Q+\mathbb{Z}]
\end{equation*}
is an isomorphism. Thus, the group $\mathcal{C}_N(D)$ is isomorphic to $\mathrm{Gal}(K_{\mathcal{O},\,N}/K)$. 
\end{proposition}
\begin{proof}
See \cite[Proposition 9.3 and Theorem 9.4]{J-K-S-Y23}. 
\end{proof}

We call the group $\mathcal{C}_N(D)$ the \textit{extended form class group of discriminant $D$ and level $N$}. 

\begin{remark}\label{diagramremark}
\begin{enumerate}
\item[(i)] Proposition \ref{CDCO} shows that $\mathcal{C}_N(D)$ generalizes the classical Gauss's form class group $\mathcal{C}_1(D)$
(see \cite[Theorem 7.7]{Cox}). A form $Q=ax^2+bxy+cy^2\in\mathcal{Q}(D,\,1)$ is said to be \textit{reduced}
if 
\begin{equation*}
|b|\leq a\leq c\quad\textrm{and}\quad
b\geq0~\textrm{if either}~|b|=a~\textrm{or}~a=c.
\end{equation*}
From the identity $b^2-4ac=D$, it follows that if $Q$ is reduced, then 
\begin{equation*}
a\leq\sqrt{\frac{|D|}{3}}. 
\end{equation*}
It is well known that if $Q_1,\,Q_2,\,\ldots,\,Q_h$ are all the reduced forms of discriminant $D$, then 
\begin{equation*}
|\mathcal{C}_1(D)|=h\quad\textrm{and}\quad
\mathcal{C}_1(D)=\{[Q_1],\,[Q_2],\,\ldots,\,[Q_h]\}
\end{equation*}
(see \cite[Theorem 2.8]{Cox}). 
\item[(ii)] For each positive integer $M$ divisible by $N$, we obtain the following commutative diagram\,:
\begin{figure}[H]
\begin{equation*}
\xymatrixcolsep{5pc}
\xymatrix{
\mathcal{C}_M(D) \ar@{->}[r]^{\sim}\ar@{->}[r]^{\sim} 
\ar@{->}[dd]_{\textrm{the natural map}}
&\mathcal{C}_M(\mathcal{O}) \ar@{->}[r]^{\sim}\ar@{->}[dd]_{\textrm{the natural map}} &
\mathrm{Gal}(K_{\mathcal{O},\,M}/K) \ar@{->>}[dd]_{\textrm{the restriction}} \\\\
\mathcal{C}_N(D) \ar@{->}[r]^\sim & \mathcal{C}_N(\mathcal{O}) \ar@{->}[r]^{\sim}
&
\mathrm{Gal}(K_{\mathcal{O},\,N}/K)
}
\end{equation*}
\caption{A diagram of class groups}
\label{diagram}
\end{figure}
It follows that the first vertical map
$\mathcal{C}_M(D)\rightarrow\mathcal{C}_N(D)$ in Figure \ref{diagram} is a surjective homomorphism. 
\end{enumerate}
\item[(iii)] The commutative diagram in Figure \ref{diagram} 
provides a geometric interpretation of the inverse limit of CM points 
on some Shimura varieties as the Galois group $\mathrm{Gal}(K^\mathrm{ab}/\mathbb{Q})$.
See the recent work of Jung-Koo-Shin \cite{J-K-S25} for further details. 
\end{remark}

\section {Actions on special values of modular functions}\label{sect4}

Following \cite{J-K-S-Y23}, we shall further adopt the perspective that the extended form class group $\mathcal{C}_N(D)$
is isomorphic to the Galois group $\mathrm{Gal}(K_{\mathcal{O},\,N}/K)$. 
\par
Let $\mathcal{F}_N$ denote the field of meromorphic modular functions of level $N$
whose Fourier coefficients lie in the $N$th cyclotomic field $\mathbb{Q}(\zeta_N)$,
where $\zeta_N=e^{2\pi\mathrm{i}/N}$. Then $\mathcal{F}_N$ is
a Galois extension of $\mathcal{F}_1$ whose Galois group is isomorphic to 
$\mathrm{GL}_2(\mathbb{Z})/\{\pm I_2\}$ (see \cite[Theorem 3 in Chapter 6]{Lang} and \cite[Propositions 6.9 (1)]{Shimura}). 
More precisely, let $f\in\mathcal{F}_N$ have Fourier expansion
\begin{equation*}
f(\tau)=\sum_{n\gg-\infty}c_nq^{n/N}\quad\textrm{for some}~c_n\in\mathbb{Q}(\zeta_N)
\quad(\tau\in\mathbb{H},~q=e^{2\pi\mathrm{i}\tau}), 
\end{equation*}
and let $\alpha\in\mathrm{GL}_2(\mathbb{Z}/N\mathbb{Z})/\{\pm I_2\}$. 
If $\alpha\in\mathrm{SL}_2(\mathbb{Z}/N\mathbb{Z})/\{\pm I_2\}$, then its action on $f$ is given by 
\begin{equation*}
f^\alpha=f\circ\widetilde{\alpha}\quad\textrm{for any lift $\widetilde{\alpha}\in\mathrm{SL}_2(\mathbb{Z})$ 
of $\alpha$}. 
\end{equation*}
If $\alpha=\left[\begin{bmatrix}1&0\\0&d\end{bmatrix}\right]$ for some $d\in(\mathbb{Z}/N\mathbb{Z})^\times$, then 
its action on $f$ is given by 
\begin{equation*}
f^\alpha(\tau)=\sum_{n\gg-\infty}c_n^{\sigma_d}q^{n/N}
\quad\textrm{where $\sigma_d\in\mathrm{Gal}(\mathbb{Q}(\zeta_N)/\mathbb{Q})$
is defined by $\zeta_N\mapsto\zeta_N^d$}.  
\end{equation*}
\par 
Define the element $\tau_\mathcal{O}\in K\cap\mathbb{H}$ by
\begin{equation}\label{tau_O}
\tau_\mathcal{O}=\left\{
\begin{array}{cl}
\displaystyle\frac{\sqrt{D}}{2} & \textrm{if}~D\equiv0\Mod{4},\\
\displaystyle\frac{-1+\sqrt{D}}{2} & \textrm{if}~D\equiv1\Mod{4}
\end{array}\right.
\end{equation}
so that $\mathcal{O}=\mathbb{Z}\tau_\mathcal{O}+\mathbb{Z}$.  
Furthermore, we let
\begin{equation}\label{irr(t)}
\mathrm{irr}(\tau_\mathcal{O},\,\mathbb{Q})=x^2+b_\mathcal{O}x+c_\mathcal{O}\in\mathbb{Z}[x]. 
\end{equation}
Recall from Shimura's theory of canonical models for modular curves that 
\begin{equation*}
K_{\mathcal{O},\,N}=K(f(\tau_\mathcal{O})~|~f\in\mathcal{F}_N~\textrm{is finite at}~\tau_\mathcal{O})
\end{equation*}
(cf. \cite[Theorem 6.2.3]{Schertz}). 

\begin{proposition}
The map
\begin{align*}
\mathcal{C}_N(D)&\rightarrow\mathrm{Gal}(K_{\mathcal{O},\,N}/K)\\
[Q]=[ax^2+bxy+cy^2] & \mapsto \left(f(\tau_\mathcal{O})\mapsto
f^{\left[\begin{smallmatrix}1&-a'(\frac{b+b_\mathcal{O}}{2})\\
0&a'\end{smallmatrix}\right]}(-\overline{\omega_Q})~\Bigg|~f\in\mathcal{F}_N~\textrm{is finite at}~\tau_\mathcal{O}\right),
\end{align*}
where $a'\in\mathbb{Z}$ satisfies $aa'\equiv1\Mod{N}$ and $\overline{\,\cdot\,}$ denotes
the complex conjugation, is a well-defined isomorphism of groups.  
\end{proposition}
\begin{proof}
See \cite[Theorem 12.3]{J-K-S-Y23}. 
\end{proof}

\section {Galois groups over ring class fields}

In this section, we
shall review Stevenhagen's results, presenting an explicit description of Shimura's reciprocity law. 
\par
Let $W_{\mathcal{O},\,N}$ be the Cartan subgroup of $\mathrm{GL}_2(\mathbb{Z}/N\mathbb{Z})$
associated with the $(\mathbb{Z}/N\mathbb{Z})$-algebra $\mathcal{O}/N\mathcal{O}$
with ordered basis $\{\tau_\mathcal{O}+N\mathcal{O},\,1+N\mathcal{O}\}$. Then we have
\begin{equation*}
W_{\mathcal{O},\,N}=\left\{\gamma=
\begin{bmatrix}t-b_\mathcal{O}s & -c_\mathcal{O}s\\s&t\end{bmatrix}~|~
s,\,t\in\mathbb{Z}/N\mathbb{Z}~\textrm{such that}~\gamma\in\mathrm{GL}_2(\mathbb{Z}/N\mathbb{Z})\right\},
\end{equation*}
where $b_\mathcal{O}$ and $c_\mathcal{O}$ are as in (\ref{irr(t)}). 
Let $H_\mathcal{O}$ denote the ring class field of order $\mathcal{O}$. 
The following concrete version of Shimura's reciprocity law is due to Stevenhagen.

\begin{proposition}\label{W/T}
If $T_{\mathcal{O},\,N}$ denotes the image
of the subgroup $\{\nu+N\mathcal{O}~|~\nu\in\mathcal{O}^\times\}$ of $(\mathcal{O}/N\mathcal{O})^\times$
under the isomorphism $(\mathcal{O}/N\mathcal{O})^\times
\stackrel{\sim}{\rightarrow} W_{\mathcal{O},\,N}$, then
there exists a well-defined isomorphism
\begin{align*}
W_{\mathcal{O},\,N}/T_{\mathcal{O},\,N}&\rightarrow\mathrm{Gal}(K_{\mathcal{O},\,N}/
H_\mathcal{O})\\
[\gamma]&\mapsto
\left(f(\tau_\mathcal{O})\mapsto
f^\gamma(\tau_\mathcal{O})
~|~f\in\mathcal{F}_N~\textrm{is finite at}~\tau_\mathcal{O}\right).
\end{align*}
\end{proposition}
\begin{proof}
See \cite[Sections 3, 4, 6]{Stevenhagen} (cf. \cite[Lemma 15.17 and Theorem 15.22]{Cox}). 
\end{proof}

\begin{remark}
\begin{enumerate}
\item[(i)] Note that
\begin{equation*}
\mathcal{O}^\times=\mathcal{O}\cap\mathcal{O}_K^\times=\left\{\begin{array}{ll}
\langle\zeta_6\rangle & \textrm{if}~D=-3,\\
\langle\zeta_4\rangle & \textrm{if}~D=-4,\\
\langle\zeta_2\rangle & \textrm{if}~D\neq-3,\,-4
\end{array}\right.
\end{equation*}
(see \cite[Exercises 4.5, 4.16, 5.9]{Cox}). 
\item[(ii)] By applying Shimura's reciprocity law, Stevenhagen 
\cite[Section 6]{Stevenhagen} also constructed an injective map $\mathcal{C}_1(D)\rightarrow\mathrm{Gal}(K_{\mathcal{O},\,N}/K)$,
extending the isomorphism $\mathcal{C}_1(D)\stackrel{\sim}{\rightarrow}\mathrm{Gal}(H_\mathcal{O}/K)$ 
arising in the theory of complex multiplication. 
Combining this with Proposition \ref{W/T}, Cho \cite{Cho} essentially established a bijection between 
$\mathcal{C}_1(D)\times W_{\mathcal{O},\,N}/T_{\mathcal{O},\,N}$ and 
$\mathrm{Gal}(K_{\mathcal{O},\,N}/K)$ (cf. \cite[Theorem 2.4]{J-K-S11} for the case $\mathcal{O}=\mathcal{O}_K$). 
\item[(iii)]
Due independently to 
Baker \cite{Baker}, 
Heegner \cite{Heegner} and 
Stark \cite{Stark}, it is known that the class number $|\mathcal{C}_1(D)|$ is one
(equivalently, $H_\mathcal{O}=K$)
if and only if 
\begin{equation*}
D\in\{-3,\,-4,\,-7,\,-8,\,-11,\,-12,\,-16,\,-19,\,-27,\,-28,\,-43,\,-67,\,-163\}. 
\end{equation*}
See also \cite[Theorems 7.24, 7.30 and 12.34]{Cox}. 
In this case, $\mathrm{Gal}(K_{\mathcal{O},\,N}/K)=\mathrm{Gal}(K_{\mathcal{O},\,N}/H_\mathcal{O})$ is
completely described
by Proposition \ref{W/T}. 
\end{enumerate}
\end{remark}

\begin{proposition}
Let $Q_0$ be the principal form in $\mathcal{Q}(D,\,N)$, that is, 
$Q_0=x^2+b_\mathcal{O}xy+c_\mathcal{O}y^2$. 
Then the map
\begin{align*}
\left\{
\left[
Q_0^\gamma\right]
\in\mathcal{C}_N(D)~|~\gamma\in\mathrm{SL}_2(\mathbb{Z})~
\textrm{such that}~Q_0^\gamma\in\mathcal{Q}(D,\,N)\right\}
&\rightarrow W_{\mathcal{O},\,N}/T_{\mathcal{O},\,N}\\
\left[Q_0^{\left[\begin{smallmatrix}
d_1&d_2\\d_3&d_4
\end{smallmatrix}\right]}\right] & \mapsto
\left[
Q_0(d_1,\,d_3)'\begin{bmatrix}
d_1 & -c_\mathcal{O}d_3\\d_3 & d_1+b_\mathcal{O}d_3
\end{bmatrix}
\right],
\end{align*}
where $Q_0(d_1,\,d_3)'$ is any integer satisfying $Q_0(d_1,\,d_3)Q_0(d_1,\,d_3)'\equiv1\Mod{N}$,
is a well-defined isomorphism of groups.   
\end{proposition}
\begin{proof}
The proof is similar to that for the case $\mathcal{O}=\mathcal{O}_K$ in \cite[Remark 10.2]{J-K-S-Y22}
(cf. \cite{Yoon}). 
\end{proof}

\section {Explicit construction of extended form class groups}

We are now ready to prove our main theorem, 
which gives a concrete description of the extended form class group
$\mathcal{C}_N(D)$. 

\begin{lemma}\label{isotropy}
Let $Q\in\mathcal{Q}(D,\,1)$. If $D\neq-3,\,-4$, then 
the isotropy group of $Q$ in $\mathrm{SL}_2(\mathbb{Z})$ is precisely $\{\pm I_2\}$. 
\end{lemma}
\begin{proof}
See \cite[Proposition 1.5 (c)]{Silverman}. 
\end{proof}

\begin{theorem}\label{main}
Let $K$ be an imaginary quadratic field and $\mathcal{O}$ be an order in $K$
with discriminant $D$. Let $Q_1,\,Q_2,\,\ldots,\,Q_h$ 
denote the complete set of reduced binary quadratic forms of discriminant $D$. For a positive integer $N$, 
let
$\gamma_1,\,\gamma_2,\,\ldots,\,\gamma_m$ be 
a complete set of representatives for the left cosets of 
the subgroup $\pm\Gamma_1(N)$ in $\mathrm{SL}_2(\mathbb{Z})$ so that
\begin{equation}\label{disjointunion}
\mathrm{SL}_2(\mathbb{Z})=\gamma_1\Gamma\cupdot\gamma_2\Gamma\cupdot
\cdots\cupdot\gamma_m\Gamma~\textrm{with}~\Gamma=\pm\Gamma_1(N). 
\end{equation}
If $D\neq-3,\,-4$, then the forms
\begin{equation}\label{representatives}
Q_i^{\gamma_k}\quad(1\leq i\leq h,~1\leq k\leq m~\textrm{such that}~Q_i^{\gamma_k}
=Q_i\left(\gamma_k\begin{bmatrix}x\\y\end{bmatrix}\right)\in\mathcal{Q}(D,\,N))
\end{equation}
represent all distinct elements of the extended form class group $\mathcal{C}_N(D)$.
\end{theorem}
\begin{proof}
Let $C\in\mathcal{C}_N(D)=\mathcal{Q}(D,\,N)/\sim_{\Gamma_1(N)}$, and so
\begin{equation}\label{CQ}
C=[Q]\quad\textrm{for some}~Q\in\mathcal{Q}(D,\,N). 
\end{equation}
By Remark \ref{diagramremark} (i), we have
\begin{equation}\label{QQ}
Q=Q_i^\alpha\quad\textrm{for some}~\alpha\in\mathrm{SL}_2(\mathbb{Z}). 
\end{equation}
From the coset decomposition (\ref{disjointunion}) of $\mathrm{SL}_2(\mathbb{Z})$, one can express
\begin{equation}\label{agb}
\alpha=\gamma_k\beta\quad\textrm{for some}~1\leq k\leq m~\textrm{and}~\beta\in\pm\Gamma_1(N). 
\end{equation}
Combining (\ref{CQ}), (\ref{QQ}) and (\ref{agb}), we obtain
\begin{equation*}
C=[Q_i^{\gamma_k\beta}]=[(Q_i^{\gamma_k\beta})^{\beta^{-1}}]=[Q_i^{\gamma_k}],  
\end{equation*}
which shows that every class in $\mathcal{C}_N(D)$ is
represented by a form of the type given in (\ref{representatives}). 
\par
Now, suppose that
\begin{equation*}
[Q_i^{\gamma_k}]=[Q_{i'}^{\gamma_{k'}}]\quad\textrm{in}~\mathcal{C}(D,\,N)
\end{equation*}
for some
$1\leq i,\,i'\leq h$ and $1\leq k,\,k'\leq m$ such that 
$Q_i^{\gamma_k},\,Q_{i'}^{\gamma_{k'}}\in\mathcal{Q}(D,\,N)$. Then there exists
$\gamma\in\Gamma_1(N)$ such that
\begin{equation*}
Q_{i'}^{\gamma_{k'}}=(Q_i^{\gamma_k})^\gamma.
\end{equation*}
This implies by Remark \ref{diagramremark} (i) that
$i'=i$ and
\begin{equation*}
Q_i^{\gamma_{k'}}=Q_i^{\gamma_k\gamma}. 
\end{equation*}
Since the isotropy group of $Q_i$ in $\mathrm{SL}_2(\mathbb{Z})$ is $\{\pm I_2\}$
by Lemma \ref{isotropy}, it follows that
\begin{equation*}
\gamma_{k'}=\gamma_k\gamma\quad\textrm{or}\quad
\gamma_{k'}=-\gamma_k\gamma,
\end{equation*}
and hence $\gamma_{k'}\Gamma$ and $\gamma_k\Gamma$ coincide. 
Therefore we conclude by (\ref{disjointunion}) that $k'=k$, showing that 
no two forms given in (\ref{representatives}) 
represent the same class in $\mathcal{C}_N(D)$. This completes the proof. 
\end{proof}

\section {Algorithm for computation}

We shall present a practical algorithm 
for computing the irreducible polynomial
for the special value of a modular function at $\tau_\mathcal{O}$. 
\par
Let $\Gamma(N)$ be the principal congruence subgroup of level $N$. 
Note that $\pm\Gamma(N)$ is normal in $\mathrm{SL}_2(\mathbb{Z})$
and $\pm\Gamma_1(N)=\langle\pm\Gamma(N),\,T\rangle$
with $T=\begin{bmatrix}1&1\\0&1\end{bmatrix}$. 
We simply write $\mathrm{SL}_2(\mathbb{Z})/\pm\Gamma_1(N)$ 
for the set of left cosets of $\pm\Gamma_1(N)$
in $\mathrm{SL}_2(\mathbb{Z})$. 

\begin{lemma}\label{pmG_1}
The following map is a bijection\,\textup{:}
\begin{align*}
\left\{
\begin{bmatrix}a\\c\end{bmatrix}\in\mathbb{Z}^2
~\Bigg|~\gcd(N,\,a,\,c)=1\right\}/\sim_\pm
&\rightarrow\mathrm{SL}_2(\mathbb{Z})/\pm\Gamma_1(N)\\
\left[\begin{bmatrix}a\\c\end{bmatrix}\right]
&\mapsto\left[\textrm{any $\gamma\in\mathrm{SL}_2(\mathbb{Z})$ satisfying}
~\gamma\equiv
\begin{bmatrix}a&\mathrm{*}\\c&\mathrm{*}\end{bmatrix}\Mod{NM_2(\mathbb{Z})}\right],
\end{align*}
where $\sim_\pm$ is the equivalence relation defined by 
\begin{equation*}
\begin{bmatrix}a\\c\end{bmatrix}\sim_\pm
\begin{bmatrix}a'\\c'\end{bmatrix}
\quad\Longleftrightarrow\quad
\begin{bmatrix}a\\c\end{bmatrix}
\equiv\begin{bmatrix}a'\\c'\end{bmatrix}~
\textrm{or}~-\begin{bmatrix}a'\\c'\end{bmatrix}\Mod{N\mathbb{Z}^2}.
\end{equation*}
\end{lemma}
\begin{proof}
See \cite[p. 104]{D-S}. 
\end{proof}

\begin{remark}
There is a one-to-one correspondence between 
$\mathrm{SL}_2(\mathbb{Z})/\pm\Gamma_1(N)$
and the set of inequivalent cusps for $\Gamma(N)$ (see \cite[Proposition 3.8.5]{D-S}). 
\end{remark}

\begin{lemma}\label{semidirect}
The field $K_{\mathcal{O},\,N}$ is Galois over $\mathbb{Q}$, and
its Galois group has the semidirect product decomposition 
\begin{equation*}
\mathrm{Gal}(K_{\mathcal{O},\,N}/\mathbb{Q})=\mathrm{Gal}(K_{\mathcal{O},\,N}/K)
\rtimes\langle\mathfrak{c}\rangle,
\end{equation*} 
where $\mathfrak{c}$ is the complex conjugation restricted to $K_{\mathcal{O},\,N}$ 
\end{lemma}
\begin{proof}
See \cite[Lemmas 6.1 and 7.1]{J-K-S-Y23-2}. 
\end{proof}

For $\gamma=\begin{bmatrix}d_1&d_2\\d_3&d_4\end{bmatrix}\in\mathrm{SL}_2(\mathbb{Z})$, we denote by 
$\widehat{\gamma}=\begin{bmatrix}d_4&d_2\\d_3&d_1\end{bmatrix}$. 

\begin{lemma}\label{bar}
If $Q\in\mathcal{Q}(D,\,1)$ and $\gamma\in\mathrm{SL}_2(\mathbb{Z})$, then we have
$-\overline{\omega_{Q^\gamma}}=\widehat{\gamma}(-\overline{\omega_Q})$ .
\end{lemma}
\begin{proof}
This follows directly from the observation $\omega_{Q^\gamma}=\gamma^{-1}(\omega_Q)$.  
\end{proof}

Combining Theorem \ref{main} with Lemmas \ref{pmG_1}, \ref{semidirect} and \ref{bar}, 
we now describe a simplified algorithm to compute
the irreducible polynomial over $\mathbb{Q}$
of the special value of a modular function at $\tau_\mathcal{O}$. 
This approach refines the earlier algorithm in Cho \cite[Theorem 8]{Cho}. 

\begin{theorem}[Algorithm for computing irreducible polynomials over $\mathbb{Q}$]\label{algorithm}
Let $K$ be an imaginary quadratic field, $\mathcal{O}$ be an order in $K$
of discriminant $D$, and $N$ be a positive integer. Let $f\in\mathcal{F}_N$ and 
$\tau_\mathcal{O}$ be as in \textup{(\ref{tau_O})}. If $D\neq-3,\,-4$ and $f$ is finite at $\tau_\mathcal{O}$, 
then the irreducible polynomial for $f(\tau_\mathcal{O})$ over $\mathbb{Q}$
can be achieved as follows\,\textup{:}
\begin{enumerate}
\item[\textup{Step 1.}] List all reduced binary quadratic forms of discriminant $D$, say,
\begin{equation*}
Q_1,\,Q_2,\,\ldots,\,Q_h. 
\end{equation*}
\item[\textup{Step 2.}] Using \textup{Lemma \ref{pmG_1}}, 
determine a complete set of representatives 
\begin{equation*}
\gamma_1,\,\gamma_2,\,\ldots,\,\gamma_m 
\end{equation*}
for the left cosets of $\pm\Gamma_1(N)$ in $\mathrm{SL}_2(\mathbb{Z})$. 
\item[\textup{Step 3.}] Write
\begin{equation*}
Q_i^{\gamma_k}=a_{i,\,k}\,x^2+b_{i,\,k}\,xy+c_{i,\,k}\,y^2\quad(1\leq i\leq h,~1\leq k\leq m),  
\end{equation*}
and define the set
\begin{equation*}
S=\{(i,\,k)~|~1\leq i\leq h,~1\leq k\leq m,~\gcd(a_{i,\,k},\,N)=1\}. 
\end{equation*}
\item[\textup{Step 4.}] Construct the polynomial 
\begin{equation*}
p(x)=\prod_{(i,\,k)\in S}\left(x-f^{\left[\begin{smallmatrix}
1 & -a_{i,\,k}'(\frac{b_{i,\,k}+b_\mathcal{O}}{2})\\
0 & a_{i,\,k}'
\end{smallmatrix}\right]\widehat{\gamma_k}}(-\overline{\omega_{Q_i}})\right)
\left(x-\overline{f^{\left[\begin{smallmatrix}
1 & -a_{i,\,k}'(\frac{b_{i,\,k}+b_\mathcal{O}}{2})\\
0 & a_{i,\,k}'
\end{smallmatrix}\right]\widehat{\gamma_k}}(-\overline{\omega_{Q_i}})}\,\,\right)
\end{equation*}
in $\mathbb{Q}[x]$, where $a_{i,\,k}'\in\mathbb{Z}$ satisfies $a_{i,\,k}a_{i,\,k}'\equiv1\Mod{N}$
and $b_\mathcal{O}$ is as in \textup{(\ref{irr(t)})}. 
\item[\textup{Step 5.}] 
Conclude that 
\begin{equation*}
p(x)=\mathrm{irr}(f(\tau_\mathcal{O}),\,\mathbb{Q})^\ell\quad\textrm{for some positive integer $\ell$}.
\end{equation*}
That is, $\mathrm{irr}(f(\tau_\mathcal{O}),\,\mathbb{Q})$ is the unique monic irreducible factor of $p(x)$. 
\end{enumerate}
\end{theorem}

\begin{remark}\label{real}
If $f(\tau_\mathcal{O})$ is real, then 
\begin{equation*}
[\mathbb{Q}(f(\tau_\mathcal{O})):\mathbb{Q}]=\frac{[K(f(\tau_\mathcal{O})):\mathbb{Q}(f(\tau_\mathcal{O}))]\cdot
[\mathbb{Q}(f(\tau_\mathcal{O})):\mathbb{Q}]}{[K:\mathbb{Q}]}
=[K(f(\tau_\mathcal{O})):K].
\end{equation*}
In this case, the irreducible polynomial for $f(\tau_\mathcal{O})$
over $\mathbb{Q}$ coincides with that over $K$. 
Thus, the polynomial $p(x)$ in Step 4 of Theorem \ref{algorithm}
may be replaced by 
\begin{equation*}
p(x)=\prod_{(i,\,k)\in S}\left(x-f^{\left[\begin{smallmatrix}
1 & -a_{i,\,k}'(\frac{b_{i,\,k}+b_\mathcal{O}}{2})\\
0 & a_{i,\,k}'
\end{smallmatrix}\right]\widehat{\gamma_k}}(-\overline{\omega_{Q_i}})\right).
\end{equation*}
\end{remark}

\section {An example}

In this final section, we shall give a concrete example by
applying the algorithm of Theorem \ref{algorithm} to the Rogers-Ramanujan continued fraction. 

\begin{example}\label{Example}
As shown by Rogers \cite{Rogers}, 
the Rogers-Ramanujan continued fraction $r$ can be written as
\begin{equation}\label{Rogers}
r(\tau)=q^{1/5}\prod_{n=1}^\infty\frac{(1-q^{5n-1})(1-q^{5n-4})}{(1-q^{5n-2})(1-q^{5n-3})}\quad(\tau\in\mathbb{H}).  
\end{equation}
Observe that $r$ has rational Fourier coefficients.  
Consider the case where $D=-52$ and $N=5$. 
In this setting, $\mathcal{O}$ is the order in
the imaginary quadratic field $K=\mathbb{Q}(\sqrt{-13})$ with conductor $\ell_\mathcal{O}=2$, and $\tau_\mathcal{O}=\sqrt{-13}$.  
Since $j(\tau_\mathcal{O})$ is an algebraic integer (see \cite[Theorem 4 in Chapter 5]{Lang})
and $x=r(\tau_\mathcal{O})$ 
satisfies the icosahedral equation (\ref{icosahedral}),  
$r(\tau_\mathcal{O})$ is a unit whose irreducible polynomial over $\mathbb{Q}$ 
has integer coefficients. 
We now apply the steps outlined in Theorem \ref{algorithm}. 
\begin{enumerate}
\item[Step 1.]
There are exactly two reduced forms of discriminant $D=-52$\,:
\begin{equation*}
Q_1=x^2+13y^2\quad\textrm{and}\quad Q_2=2x^2+2xy+7y^2. 
\end{equation*}
\item[Step 2.]
Using Lemma \ref{pmG_1}, we determine a complete set of representatives 
of $\pm\Gamma_1(N)$ in $\mathrm{SL}_2(\mathbb{Z})$ as follows\,:
\begin{equation*}
\begin{array}{llll}
\gamma_1=\begin{bmatrix}0&-1\\1&0\end{bmatrix},
&\gamma_2=\begin{bmatrix}5&-3\\2&-1\end{bmatrix},
&\gamma_3=\begin{bmatrix}1&0\\0&1\end{bmatrix},
&\gamma_4=\begin{bmatrix}1&0\\1&1\end{bmatrix},\vspace{0.1cm}\\
\gamma_5=\begin{bmatrix}1&0\\2&1\end{bmatrix},
&\gamma_6=\begin{bmatrix}1&0\\3&1\end{bmatrix},
&\gamma_7=\begin{bmatrix}1&0\\4&1\end{bmatrix},
&\gamma_8=\begin{bmatrix}2&-1\\5&-2\end{bmatrix},\vspace{0.1cm}\\
\gamma_9=\begin{bmatrix}2&-1\\1&0\end{bmatrix},
&\gamma_{10}=\begin{bmatrix}2&-1\\7&-3\end{bmatrix},
&\gamma_{11}=\begin{bmatrix}2&1\\3&2\end{bmatrix},
&\gamma_{12}=\begin{bmatrix}2&-1\\9&-4\end{bmatrix}.
\end{array}
\end{equation*}
\item[Step 3.] We find that
{\small\begin{equation*}
\begin{array}{lll}
Q_1^{\gamma_1}=13x^2+y^2, & Q_1^{\gamma_2}=77x^2-82xy+22y^2, & Q_1^{\gamma_3}=x^2+13y^2,\\
Q_1^{\gamma_4}=14x^2+26xy+13y^2, & Q_1^{\gamma_5}=53x^2+52xy+13y^2, & Q_1^{\gamma_6}=118x^2+78xy+13y^2,\\
Q_1^{\gamma_7}=209x^2+104xy+13y^2, & Q_1^{\gamma_8}=329x^2-264xy+53y^2, & 
Q_1^{\gamma_9}=17x^2-4xy+y^2,\\
Q_1^{\gamma_{10}}=641x^2-550xy+118y^2, & Q_1^{\gamma_{11}}=121x^2+160xy+53y^2, & 
Q_1^{\gamma_{12}}=1057x^2-940xy+209y^2,\\
Q_2^{\gamma_1}=7x^2-2xy+2y^2, & Q_2^{\gamma_2}=98x^2-110xy+31y^2, & Q_2^{\gamma_3}=2x^2+2xy+7y^2,\\
Q_2^{\gamma_4}=11x^2+16xy+7y^2, & Q_2^{\gamma_5}=34x^2+30xy+7y^2, & 
Q_2^{\gamma_6}=71x^2+44xy+7y^2,\\
Q_2^{\gamma_7}=122x^2+58xy+7y^2, & 
Q_2^{\gamma_8}=203x^2-166xy+34y^2, & 
Q_2^{\gamma_9}=19x^2-10xy+2y^2,\\
Q_2^{\gamma_{10}}=379x^2-328xy+71y^2, & 
Q_2^{\gamma_{11}}=83x^2+106xy+34y^2, & 
Q_2^{\gamma_{12}}=611x^2-546xy+122y^2.
\end{array}
\end{equation*}}
Moreover, we justify that every pair $(i,\,k)$ with $1\leq i\leq 2$ and $1\leq k\leq 12$ satisfies
$\gcd(a_{i,\,k},\,N)=1$, and hence
\begin{equation*}
S=\{(i,\,k)~|~1\leq i\leq 2,~1\leq k\leq12\}.  
\end{equation*}
\item[Step 4.]
Since $r(\tau_\mathcal{O})=r(\sqrt{-13})$ is real, 
we may simplify the polynomial $p(x)$ as in 
Remark \ref{real}\,:
\begin{equation*}
p(x)=\prod_{(i,\,k)\in S}\left(x-r^{\left[\begin{smallmatrix}
1 & -a_{i,\,k}'(\frac{b_{i,\,k}+b_\mathcal{O}}{2})\\
0 & a_{i,\,k}'
\end{smallmatrix}\right]\widehat{\gamma_k}}(-\overline{\omega_{Q_i}})\right).
\end{equation*}
And, we compute that
\begin{align*}
p(x)=&\prod_{(i,\,k)\in S}\left(x-r^{\left[\begin{smallmatrix}
1 & -a_{i,\,k}'(\frac{b_{i,\,k}+b_\mathcal{O}}{2})\\
0 & 1
\end{smallmatrix}\right]\widehat{\gamma_k}}(-\overline{\omega_{Q_i}})\right)\\
&\hspace{5cm}\textrm{because $r$ has rational Fourier coefficients}\\
=&\prod_{(i,\,k)\in S}\left(x-r\left(\gamma_{i,\,k}(-\overline{\omega_{Q_i}})\right)\right)
\quad\textrm{where $\gamma_{i,\,k}$ is any matrix in $\mathrm{SL}_2(\mathbb{Z})$ such that}\\
&\hspace{5cm}\gamma_{i,\,k}\equiv
\left[\begin{smallmatrix}
1 & -a_{i,\,k}'(\frac{b_{i,\,k}+b_\mathcal{O}}{2})\\
0 & 1
\end{smallmatrix}\right]\widehat{\gamma_k}\Mod{5M_2(\mathbb{Z})}\\
&\hspace{5cm}\textrm{because $r$ is modular for $\Gamma(5)$}\\
=&\left(x-r\left(\left[\begin{smallmatrix}0&-1\\1&0\end{smallmatrix}\right](\sqrt{-13})\right)\right)
\left(x-r\left(\left[\begin{smallmatrix}-5&-8\\-3&-5\end{smallmatrix}\right](\sqrt{-13})\right)\right)
\left(x-r\left(\left[\begin{smallmatrix}1&0\\0&1\end{smallmatrix}\right](\sqrt{-13})\right)\right)\\
&\times\left(x-r\left(\left[\begin{smallmatrix}-1&-2\\1&1\end{smallmatrix}\right](\sqrt{-13})\right)\right)
\left(x-r\left(\left[\begin{smallmatrix}-3&-2\\2&1\end{smallmatrix}\right](\sqrt{-13})\right)\right)
\left(x-r\left(\left[\begin{smallmatrix}-3&2\\-2&1\end{smallmatrix}\right](\sqrt{-13})\right)\right)\\
&\times\left(x-r\left(\left[\begin{smallmatrix}-1&2\\-1&1\end{smallmatrix}\right](\sqrt{-13})\right)\right)
\left(x-r\left(\left[\begin{smallmatrix}3&-5\\5&-8\end{smallmatrix}\right](\sqrt{-13})\right)\right)
\left(x-r\left(\left[\begin{smallmatrix}1&1\\1&2\end{smallmatrix}\right]\sqrt{-13})\right)\right)\\
&\times\left(x-r\left(\left[\begin{smallmatrix}2&-1\\-3&2\end{smallmatrix}\right]\sqrt{-13})\right)\right)
\left(x-r\left(\left[\begin{smallmatrix}2&1\\3&2\end{smallmatrix}\right](\sqrt{-13})\right)\right)
\left(x-r\left(\left[\begin{smallmatrix}1&-1\\-1&2\end{smallmatrix}\right](\sqrt{-13})\right)\right)\\
&\times\left(x-r\left(\left[\begin{smallmatrix}-2&-1\\1&0\end{smallmatrix}\right](\tfrac{1+\sqrt{-13}}{2})\right)\right)
\left(x-r\left(\left[\begin{smallmatrix}-1&2\\2&-5\end{smallmatrix}\right](\tfrac{1+\sqrt{-13}}{2})\right)\right)
\left(x-r\left(\left[\begin{smallmatrix}1&2\\0&1\end{smallmatrix}\right](\tfrac{1+\sqrt{-13}}{2})\right)\right)\\
&\times\left(x-r\left(\left[\begin{smallmatrix}-2&-3\\1&1\end{smallmatrix}\right](\tfrac{1+\sqrt{-13}}{2})\right)\right)
\left(x-r\left(\left[\begin{smallmatrix}1&0\\2&1\end{smallmatrix}\right](\tfrac{1+\sqrt{-13}}{2})\right)\right)
\left(x-r\left(\left[\begin{smallmatrix}5&-2\\-2&1\end{smallmatrix}\right](\tfrac{1+\sqrt{-13}}{2})\right)\right)\\
&\times\left(x-r\left(\left[\begin{smallmatrix}-2&3\\-1&1\end{smallmatrix}\right](\tfrac{1+\sqrt{-13}}{2})\right)\right)
\left(x-r\left(\left[\begin{smallmatrix}-2&1\\-5&2\end{smallmatrix}\right](\tfrac{1+\sqrt{-13}}{2})\right)\right)
\left(x-r\left(\left[\begin{smallmatrix}0&-1\\1&2\end{smallmatrix}\right](\tfrac{1+\sqrt{-13}}{2})\right)\right)\\
&\times\left(x-r\left(\left[\begin{smallmatrix}-1&1\\-3&2\end{smallmatrix}\right](\tfrac{1+\sqrt{-13}}{2})\right)\right)
\left(x-r\left(\left[\begin{smallmatrix}-1&-1\\-2&-3\end{smallmatrix}\right](\tfrac{1+\sqrt{-13}}{2})\right)\right)
\left(x-r\left(\left[\begin{smallmatrix}-2&5\\-1&2\end{smallmatrix}\right](\tfrac{1+\sqrt{-13}}{2})\right)\right)\\
=&~x^{24}
+82\,x^{23}
-996\,x^{22}
+968\,x^{21}
+1051\,x^{20}
+1422\,x^{19}
-96\,x^{18}\\
&-24912\,x^{17}
+7896\,x^{16}
+16722\,x^{15}
+28844\,x^{14}
+13658\,x^{13}
-114024\,x^{12}\\
&-13658\,x^{11}
+28844\,x^{10}
-16722\,x^{9}
+7896\,x^{8}
+24912\,x^{7}
-96\,x^{6}\\
&-1422\,x^{5}
+1051\,x^{4}
-968\,x^{3}
-996\,x^{2}
-82\,x
+1\quad\textrm{by (\ref{Rogers}) and numerical estimation}, 
\end{align*}
which shows that $r(\tau_\mathcal{O})$ is a unit as desired. 
\item[Step 5.] The above $p(x)$ 
turns out to be irreducible over $\mathbb{Q}$, and hence we get
$\mathrm{irr}(r(\tau_\mathcal{O}),\,\mathbb{Q})=p(x)$. 
\end{enumerate}
\end{example}

\begin{remark}
\begin{enumerate}
\item[(i)] For $\mathbf{r}=\begin{bmatrix}r_1 & r_2\end{bmatrix}\in M_{1,\,2}(\mathbb{Q})\setminus
M_{1,\,2}(\mathbb{Z})$, the Klein form $\mathfrak{k}_\mathbf{r}$ is defined by the infinite product
\begin{equation*}
\mathfrak{k}_\mathbf{r}(\tau)=
e^{\pi\mathrm{i}r_2(r_1-1)}q^{r_1(r_1-1)/2}(1-q_z)\prod_{n=1}^\infty(1-q^nq_z)(1-q^nq_z^{-1})(1-q^n)^{-2}
\quad(\tau\in\mathbb{H}), 
\end{equation*}
where $q_z=e^{2\pi\mathrm{i}z}$ with $z=r_1\tau+r_2$. 
Modularity and other fundamental properties are established in 
\cite[Sections 2.1 and 3.7]{K-L}. 
By comparing with
(\ref{Rogers}), one deduces the identity
\begin{equation}\label{relation}
r(\tau)=\zeta_5\prod_{m=0}^4\frac{\mathfrak{k}_{\left[\begin{smallmatrix}
\frac{1}{5} & \frac{m}{5}\end{smallmatrix}\right]}(\tau)}
{\mathfrak{k}_{\left[\begin{smallmatrix}\frac{2}{5} & \frac{m}{5}\end{smallmatrix}\right]}(\tau)}
=\frac{\mathfrak{k}_{\left[\begin{smallmatrix}\frac{1}{5} & 0\end{smallmatrix}\right]}(5\tau)}
{\mathfrak{k}_{\left[\begin{smallmatrix}\frac{2}{5} & 0\end{smallmatrix}\right]}(5\tau)}\quad
(\tau\in\mathbb{H}). 
\end{equation}
\item[(ii)] For an alternative approach to estimating the Rogers-Ramanujan continued fraction, see
\cite{G-H}. 
\end{enumerate}
\end{remark}

\section*{ Acknowledgements}

The Maple code used to perform the computations in Example \ref{Example} is available on the corresponding author's website\,:
\url{https://sites.google.com/site/mathdsyoon/home/research}.

\bibliographystyle{amsplain}

\address{
Department of Mathematical Sciences \\
KAIST \\
Daejeon 34141\\
Republic of Korea} {jkgoo@kaist.ac.kr}
\address{
Department of Mathematics\\
Hankuk University of Foreign Studies\\
Yongin-si, Gyeonggi-do 17035\\
Republic of Korea} {dhshin@hufs.ac.kr}
\address{ 
Department of Mathematics Education\\
Pusan National University\\
Busan 46241\\Republic of Korea}
{dsyoon@pusan.ac.kr}
\end{document}